\documentclass[11pt]{amsart}

\usepackage{amsmath,amsthm,amssymb,mathrsfs,amsfonts,verbatim}
\usepackage{color}

\usepackage[all,tips]{xy}
\usepackage{graphicx,ifpdf}
\ifpdf
\fi

\newtheorem{theorem}{Theorem}[section]
\newtheorem{lemma}[theorem]{Lemma}
\newtheorem{corollary}[theorem]{Corollary}
\newtheorem{proposition}[theorem]{Proposition}

\theoremstyle{definition}
\newtheorem{definition}[theorem]{Definition}
\newtheorem{example}[theorem]{Example}

\theoremstyle{remark}
\newtheorem{remark}[theorem]{Remark}

\numberwithin{equation}{section}

\begin{document}

\title{Convergence of measures under diagonal actions on homogeneous spaces}

\author{Ronggang Shi}
\address{School of Mathematical Sciences, Xiamen University, Xiamen 361005, PR China}
 \email{ronggang@xmu.edu.cn}
\thanks{The author was supported in part by the Fundamental Research Funds for the Central Universities(2010121004) of China.}


\date{}



\begin{abstract}
\noindent
Let $\lambda$ be a probability measure on $\mathbb T^{n-1}$ where $n=2$ or $3$.
Suppose $\lambda$ is invariant, ergodic and has positive entropy
with respect to
the linear transformation defined by a hyperbolic matrix.
We get a measure $\mu $ on $SL_n(\mathbb Z)\backslash SL_n(\mathbb R)$
by putting  $\lambda$ on some unstable horospherical orbit of the right
translation of $a_t=\mathrm{diag}(e^t, \ldots, e^t, e^{-(n-1)t})$ $(t>0)$.
We prove that if the average of $\mu$ with respect to the flow
$a_t$ has a limit, then
it must be a scalar  multiple of the probability Haar measure.
As an application we show that if the entropy of $\lambda$ is large,
then Dirichlet's theorem is not improvable $\lambda$ almost surely.
\end{abstract}

\maketitle

\markright{CONVERGENCE OF MEASURES}

\section{Introduction}\label{introduction}
For any positive integer $p$, let $\times p$ be the map
of $[0,1]=\mathbb Z\backslash \mathbb R$ to
itself which sends $x\mod \mathbb Z$ to $px\mod \mathbb Z$.
Let  $q$ be another positive integer such that $p$ and $q$ are
 not powers of any
integer, then
it is conjectured by Furstenburg that the only
non-atomic $\times p$ and $\times q$  ergodic probability  measure on
$[0,1]$ is the Lebesgue measure.
So far the best results known are obtained by Rudolph \cite{Ru} and Johnson \cite{Jo}
under the assumption of positive entropy with respect to
the  $\times p$ map. Using their
measure rigidity results they proved in \cite{JR95} that
the average of  the $\times q$ orbit of
 any $\times p$ ergodic  probability measure
with positive entropy converges to the Lebesgue measure.
Here and hereafter all the convergence of
measures on some locally compact topological space
$X$
are under the weak$^*$ topology of
Radon measures $\mathcal M(X)$.

Recently, the measure rigidity  on torus is generalized
 to  homogeneous spaces by Einsieder, Katok and Lindenstrauss,
see \cite{EKL}, \cite{L06} and \cite{EL08} for details.
Using these new results  we can extend
the convergence of measures in \cite{JR95} to homogeneous
spaces.

For a lattice $\Gamma$ of a locally compact group $G$,
we use $m_{\Gamma \backslash G}$ to denote the unique
probability Haar measure on $\Gamma \backslash G$ and use
$\vartheta_{\Gamma \backslash G}$ to denote
the element $\Gamma$ in $\Gamma \backslash G$.
An element $g\in G$ also stands for
the map on $\Gamma \backslash G$ which sends $x\in
\Gamma \backslash G$ to $xg$. Let $\mathbb R\to G$
be a continuous homomorphism and $a_t$ be the image of $t\in
\mathbb R$. We consider the one-parameter subgroup $\{a_t\}$
as a flow on $\Gamma \backslash G$ through right translation.
For a Radon measure
$\mu$ on
$\Gamma \backslash G$, let $a_t\mu$ be the pushforward of
$\mu$ under  the map $a_t$, that is
 $a_t\mu (C)=\mu(Ca_t^{-1})$ for any Borel measurable
 subset
$C\subset \Gamma \backslash G$.
We say  $\mu$    has non-escape of mass with respect to the
flow $a_t$, if any
limit of
\begin{equation}\label{eq:set of mea}
\left\{ {\color{red}\frac{1}{T}}\int_0^{T} a_t \mu\,
\mathrm d t: T\ge 0\right \}
\end{equation}
 is a probability measure.

We consider Euclidean space $\mathbb R^{n-1}$ as
the space of row
vectors.
For any $\mathbf s\in \mathbb R^{n-1}$,
take
\begin{equation}\label{eq:Intro:us}
u(\mathbf s)=\left(\begin{array}{cc}
\mathrm{Id}_{n-1} & 0 \\
\mathbf s & 1
\end{array}
\right)\in SL_n(\mathbb R).
  \end{equation}

\begin{theorem}\label{thm:intro2}
Let $\lambda$ be the probability Cantor measure on $[0,1]$.
 Suppose
 $a_t=\mathrm{diag}
(e^{t}, e^{-t})\in SL_2(\mathbb R)$ and
$Y=SL_2(\mathbb Z)\backslash SL_2(\mathbb R).$  Then for any $f\in C_c(Y)$,
\[
\lim_{T\to \infty}\frac{1}{T}\int_0^T
\int_0^1 f(\vartheta_Y u(s)a_t) \,
\mathrm d \lambda (s)
 \, \mathrm d t
=\int_Y f \, \mathrm dm_Y.
\]
\end{theorem}
The above theorem strengthens Theorem 1.6 of \cite{EFS}
where it is proved that for $\lambda$ almost every ${\color{red}s\in [0,1]}$ the
trajectory $\vartheta _Y u(s)a_t$ with $t\in \mathbb R_{\ge 0}$ is dense.
 A general version
of Theorem \ref{thm:intro2}  will be proved in Section \ref{sec:example}.
The idea  is to lift the flow $a_t$ and the measure on $Y$
to an irreducible  quotient of product of real and $p$-adic Lie groups
 which follows \cite{EFS}, then we prove
 equidistribution  there.

 To state this result, we need to introduce
some notation.
Let $S=\{\infty, p_1, \cdots, p_l\}$ where $p_i$ is prime and $l>0$.
Take
 $\mathbb Q_S=\prod_{v\in S}\mathbb Q_v$
 where $\mathbb Q_v$ is the completion of $\mathbb Q$
 at the place $v$
  and
integer $n=p_1^{\sigma_1}\cdots p_l^{\sigma_l}$ where $\sigma_i>0$.
For $\mathbf s\in \mathbb Q_S$, we define
$u(\mathbf s)\in SL_2(\mathbb Q_S)$ as in (\ref{eq:Intro:us}). 
In this paper we consider elements of $GL_2$ acts
on quotients of $PGL_2$ through the natural projection map and
use elements of $GL_2$ to represent their image in $PGL_2$.

\begin{theorem}\label{thm:intro0}
Let $\mu$ be a probability measure on
\[
X=PGL_2(\mathbb Z[1/p_1,\ldots, 1/p_l])\backslash PGL_2(\mathbb Q_S)
\]
supported
on
$
\vartheta_X u(\mathbb Q_S)
$.
Suppose  $\mu$ is
invariant,  ergodic and has positive entropoy
  with respect to the
right translation of $\mathrm{diag}(n,1)^{l+1}$.
If $\mu$ has non-escape of mass with respect
to the flow $a_t=\mathrm{diag}(e^t,e^{-t})$ of the $PGL_2(\mathbb R)$
factor, then
$
\lim_{T\to \infty}\frac{1}{T}\int_0^T
 a_t \mu
 \, \mathrm d t
$
is an ${\color{red}PGL_2(\mathbb R)^\circ}$ invariant homogeneous measure.
\end{theorem}

Theorem \ref{thm:intro0} is based on the Theorem
\ref{thm:main}, which shows that  all the ergodic components
of the limit of (\ref{eq:set of mea})
with respect to the flow $a_t$
have positive entropy. Then we apply
measure rigidity in Lindenstrauss \cite{L06}.
In fact Theorem
\ref{thm:main} is a general result about entropy and can also
be  applied to the homogeneous space
$Z=SL_3(\mathbb Z)\backslash SL_3(\mathbb R)$ even in the
case that  there is some mass left.

 Elements
of $SL(2, \mathbb Z)$ acts naturally on
 $\mathbb T^2=
\mathbb Z^2\backslash \mathbb R^2$
 from right by
matrix multiplications.
We can put a probability measure $\lambda$ of $\mathbb T^2$
on $Z$ to get a measure $\mu$ in the following way: for any $f\in C_c(Z)$,
\begin{equation}\label{eq:sl3induced}
\int_{\mathbb T^2}f(\vartheta_Z
 u(\mathbf s))\, \mathrm d \lambda(\mathbf s)
=\int_Z f\, \mathrm d\mu.
\end{equation}

\begin{theorem}\label{thm:intronew}
Let $\gamma\in SL_2(\mathbb Z)$ be a hyperbolic matrix and let
 $\lambda$ be an $\gamma$-invariant
and ergodic probability measure on $\mathbb T^2$ with positive entropy.
Suppose $a_t=\mathrm{diag}(e^t, e^t, e^{-2t})$ and $\mu$ is the measure
on $Z$ induced from $\lambda$ by (\ref{eq:sl3induced}).
If for a sequence of real numbers $(T_n)$ with $T_n\to \infty$,
 the limit of
 \begin{equation}\label{eq:in intronew}
\mu_n=\frac{1}{T_n}\int_0^{T_n}
a_t \mu
 \, \mathrm d t
\end{equation}
exists, then $\lim \mu_n=cm_Z$ where $0\le c\le  1$.
\end{theorem}

We remark here that a large class of measures on $[0,1]^2$
such as friendly measures of \cite{KLW} satisfy
$c=1$. Also by a theorem of
\cite{EK}, if the entropy of $\lambda$ with respect to
$\gamma$ is large, then $c>0$.
The equidistribution result
of Theorem \ref{thm:intronew} has the following number theoretical
applications.

\begin{theorem}\label{thm:main diophantine}
Let $\gamma\in SL_2(\mathbb Z)$ be a hyperbolic matrix.
Suppose $\lambda$ is a $\gamma$-invariant and ergodic
 probability measure on $\mathbb T^2$ such that
 $h_\lambda(\gamma)=c\,h_m(\gamma)$ where $m$ is
 the Lebesgue measure and $c>\frac{2}{3}$, then Dirichlet's
 theorem can not be improved
 (see Section \ref{sec:Diophantine}) $\lambda$ almost surely.
\end{theorem}
\begin{remark}
Here the number $\frac{2}{3}$ comes from the mass estimate
in \cite{EK}. For $c\le \frac{2}{3}$ we don't know whether
there is always a sequence $(T_n)$ so that the limit
in Theorem \ref{thm:intronew} is nonzero.
\end{remark}

\vspace{0.5cm}
\noindent \textbf{Acknowledgements}
The author would like to thank  Manfred Einsiedler for
 many helpful conversations.

\section{Preliminaries: Notation and leaf-wise measures}
\label{sec:leafmeasure}

Let $L$ be a product of semisimple  algebraic groups, that is
$L=G_1\times\cdots \times G_l$ where $G_i$ is a semisimple
linear
algebraic group over a local field of characteristic zero.
Let $G=SL_n(\mathbb R)\times L$ where $L$  could
be  trivial  and $n\ge 2$.
Here are some conventions that will be used throughout the paper.
 We use $\mathrm {Id}$ to denote
the identity element of various groups according to the context.
We consider $SL_n(\mathbb R)$,  $G_i$ and their products
 also as the corresponding
subgroups of $G$, and they are called factors of $G$.

Suppose a factor group $H$  of $G$ is  an algebraic group over
the local field $k$ with a fixed  absolute value. Let
$g\in H$ be a $k$-split semisimple element such that
the adjoint action $\mathrm{Ad}_{g}$ of $g$ on the
Lie algebra of $H$ satisfies
\begin{itemize}
\item $1$ is the only eigenvalue of norm $1$;
\item No two different eigenvalues have the same norm.
\end{itemize}
Then following Margulis and Tomanov \cite{MT}
we say $g$ is from the class $\mathscr A$.
We remark here that our concept  of
class $\mathscr A$ element is a little bit different
from that of \cite{MT} but follows that of \cite{EL081}.
We say
$a\in G$ is an element from the class $\mathscr A$ if
all of its components
 are from the class $\mathscr A$.

We fix a non-neural  element $a\in G$ from the class
$\mathcal A$. The unstable horospherical subgroup of $a$ is
\[
G^+=\{g\in G
:a^{-n} g a^{n}\to e \mbox{ as } n\to -\infty\}.
\]
We say a subgroup $U$ of $G$ is a one-dimensional unipotent
subgroup if $U$ is an algebraic subgroup of a factor of $G$
and $U$ is isomorphic to the base field
considered as an additive algebraic group.
Let $U$ be a one dimensional  subgroup of $G^+$ normalized by $a$.
The norm on $k$ induce a metric $d^U$ on $U$.
   For $u\in U$, let $B_r^U(u)$ or $B_r^U$ if
$u$ is the identity stand for the
 ball of radius $r$ in $U$ centered at $u$.

Let $X$ be a locally compact topological space and $G$ acts on $X$
from right.
Next we
review leaf-wise measures along $U$ foliations
for  an arbitrary $a$-invariant finite positive measure $\nu$
on $X$.
More details can be found in \cite{L06} and \cite{EL081}.
Let $\mathcal M(U)$ be the set of Radon measures on $U$.
There is a measurable map $X \to \mathcal M(U)$ and the
image of
 $x\in X$ is denoted by
$\nu_x^U$.
The measure $\nu_x^U$
 is  called leaf-wise measure along
$U$ foliations and is normalized so
that $\nu_x^U(B_1^U)=1$.
 The entropy contribution of $U$ with respect to $\nu$
 is an $a$-invariant measurable function
on $X$ defined by
\begin{equation}
 h_\nu(a,U)(x)=\lim_{n\to \infty}
\frac{\log\nu_x^U(a^{-n}B_1^Ua^n)}{n}.
\end{equation}
Zero entropy contribution is characterized as follows:
\begin{theorem}[Theorem 7.6 of \cite{EL081}]\label{thm:zero entropy}
For $\nu$ almost every $x\in X$, $h_\nu(a,U)(x)=0$  if and only if $\nu_x^U$ is trivial,
that is an atomic measure concentrated on the identity.
\end{theorem}

Leaf-wise measures are constructed by pasting
various  conditional measures
with respect to  countably generated  sigma-rings whose
atoms are orbits of  open bounded subsets of $U$.
For us   sigma-ring is a collection of Borel subsets
of $X$ which are closed under countable unions
and set differences.
For example, given a finite partition $\mathcal P$ of
a set $D$, then arbitrary unions of elements of
 $\mathcal P$  is a
 sigma-ring with maximal element $D$.
In this paper we will not distinguish between a finite
partition and the  sigma-ring it generates.
Let $\mathcal A$ and $\mathcal B$ be two  sigma-rings, then
$\mathcal A\bigvee \mathcal B$ stands for the smallest  sigma-ring
containing both of them.

Let $D$ be a measurable subset of $X$ and let $\mathcal A$
be a
countably generated  sigma-ring  with maximal
element $D$. For each $x\in D$ the atom of $x$
with respect to $\mathcal A$ is
\begin{equation}\notag
[x]_{\mathcal A}=\bigcap _{x\in C\in \mathcal A}C.
\end{equation}
Let $\mathcal M(X)$ be the set of Radon measures on $X$.
We can assign measurably for each $x\in D$ a probability measure
$\nu_x^{\mathcal A}\in \mathcal M(X)$ supported on $[x]_{\mathcal A}$.
Here $\nu_x^{\mathcal A}$ is called the conditional measure of $\nu|_D$ with respect to $\mathcal A$.
 If $\mathcal A$ is
a finite partition of $D$, conditional measure is
uniquely determined by
$$\nu_x^{\mathcal A}(C)
\nu([x]_{\mathcal A})=\nu([x]_{\mathcal A}\cap C)$$ for any
Borel subset $C$ of $X$.


The  sigma-ring $\mathcal A$ is said to be subordinate to $U$
on $D$ if for $\nu$  almost every $x\in D$, we have
$[x]_{\mathcal A}=U_x.x$ where $U_x$ is an open subset of $ U$
and  there exists $\delta>0$
such that
\[B_\delta^U\subset U_x\subset B_{\delta^{-1}}^U.
\]
If $\mathcal A$ is subordinate to $U$, then leaf-wise measures on $D$ are proportional
to the conditional measures with respect to
$\mathcal A$. More precisely, let $x(\nu_x^U|_{U_x})$ be
 the measure on $X$ supported on $xU_x$ such that
 for any Borel set $V\subset U_x$ the value
 $[x(\nu_x^U|_{U_x})](xV )=
\nu_x^U(V)$, then
$x(\nu_x^U|_{U_x})$
is proportional to
$ \nu_x^{\mathcal A}$
for $\nu$  almost every $x\in D$.

\section{Entropy of the limit measure}\label{sec:keytheorem}

\begin{theorem}\label{thm:main}
Let
$\Gamma$ be a discrete subgroup of $G= SL_n(\mathbb R)\times L $
where $L$ is a product of semisimple algebraic groups.
Let $a\in G$ be an element from the class $\mathscr A$
and let
$a_t=\mathrm{diag}(e^t,\ldots,  e^{t}, e^{-(n-1)t})
\in SL_n(\mathbb R)
$ which
commutes with $a$.
Suppose the
one-dimensional subgroup
$U$ is normalized by $a$, $a_t$ and
contained in the unstable horospherical subgroup
of $a$.
Take
 $\mu$ to be an $a$-invariant and  ergodic probability measure on $\Gamma
\backslash G$
with $h_\mu(a, U)>0$.
If for a sequence of real numbers $(T_n)$ with $T_n\to \infty$ the limit
$$\nu=\lim_{n\to \infty}\frac{1}{T_n}\int_0^{T_n} a_t\mu\, \mathrm dt$$
is none-zero,
 then $h_{\nu}(a, U)(x)>0$ for $\nu$ almost every $x$.
\end{theorem}
\begin{remark}\label{rem:general case}
There is no need to assume $U$ is one dimensional if $a_t$ commutes with $U$.
But in the case where $U=U_1\times U_2$, $h_\mu(a,U_1)=h_\mu(a, U_2)=0$ and
$a_t$ centralizes $U_1$ but contracts or attracts $U_2$, the proof
given below doesn't work.
\end{remark}

We first do some preparation for the proof of Theorem \ref{thm:main}
and establish some technical lemmas. To simplify the notation we 
denote $G=G_0\times G_1 \cdots\times G_l$ where $G_0=SL_n(\mathbb R)$
and $G_i$ is an algebraic group over $k_i$ with a fixed norm
$|\cdot|_i$.

We introduce a parameterization of  neighborhoods for
points in $X$ using the Lie algebra of $G$
so that we can construct  sigma-rings
of neighborhoods of $X$
 whose
atoms are sections of $U$ orbits.
Let
$\mathfrak g_i$ be the Lie algebra of $G_i$, then
the Lie algebra of $G$ is
$\mathfrak g= \prod_{i=0}^l
\mathfrak {g}_i$. The exponential map $\exp$
from $\mathfrak g$ to $G$ is defined by
$$(\ldots,  w_i,\ldots )\in \mathfrak g  \to
(\ldots, \exp(w_i),\ldots)\in G.$$
Suppose $U$ is a subgroup of $G_j$ and the Lie algebra of
$U$ is $\mathrm{Lie}(U)$.
 Let $\tilde {\mathfrak u}$ be a $k_j$-linear  subspace
of $\mathfrak g_j$ 
{\color{red} complementary 
to $\mathrm{Lie}(U)$ with respect to the adjoint action 
of $a$ and $a_t$.
}
Take  the linear space
$\mathfrak v$ of $\mathfrak g$ to be the product of
 $\mathfrak g_i$
except for the place j where it is
the linear space $\tilde {\mathfrak u}$.
For each factor of $\mathfrak v$ we fix a basis and let $\|\cdot \|_i$
be the corresponding sup norm induced by this basis and the norm
on the base field $k_i$.
We define a norm $\|\cdot\| $ on $\mathfrak v$ by
$$\| (w_0, \ldots,w_l)\|=\sup_i (\|w\|_i).$$

Let $V=\{\exp(v): v\in \mathfrak v\}$.
The ball of radius $r$ in $V$ centered at
 $g=\exp(w_0)$ with $w_0\in \mathfrak v$ is denoted by
$$B_r^V(g)=\{\exp (w_0+w):w\in \mathfrak v,
\|w\|_{\mathfrak v}<r\}.$$
If $g$ is the identity we simply denote this ball by
$B_r^V$.
As in Section \ref{sec:leafmeasure}, we take $d^U$ to be
the    translation invariant
metric   on the unipotent subgroup $U$ induced from
the norm of the local field and $B_r^U(g)$
or $B_r^U$ if $g$ is the identity
 is the ball of radius $r$ centered at $g\in U$.
For each  $x\in X$ there exists $r>0$ such that
the map from 
 $B_r^VB_r^U=\{gh: g\in B_r^V, h\in B_r^U\}$ to $X$
 which sends $vu$ to $xvu$ is injective.
We call such a number $r$  an injectivity radius at $x$.
For technical reasons we further require that injectivity
radius $r$ is small so that the product space
$B_r^V\times B_r^U$ can be naturally identified
$B_r^VB_r^U$.

Let $\nu$ be the nonzero limit measure  in Theorem \ref{thm:main}.
Without loss of generality we assume that $T_n=n$ and $\nu=\lim \mu_n$
 where
\[
\mu_n=\frac{1}{n}\int_0^n a_t\mu\,\mathrm d t.
\]
Now we assume that the set of points $x$ with
$h_{\nu}(a, U)(x)=0$  has positive
$\nu$ measure and develop some consequences.

There exist $\sigma>0$, $z\in X$ and $r>0$
 such that $2r$ is an injectivity radius at $z$ and
 the set $D=zB_r^VB_r^U$
 has the following property:
 \begin{equation}\label{eq:corec D}
\nu\big(\{x\in D: \nu _x^U  \mbox{ is trivial} \}\big)>\sigma
\end{equation}
Let $D_2=zB_r^VB_{2r}^U$, we assume that
 the boundary of $D_2$ in $X$ has $\nu$ measure zero.
 We are going to construct a countably generated sigma-ring
 on $D_2$ subordinate to  $U$.

\begin{lemma}\label{lem:patition}
There exist  increasing sequences of finite partitions
$\mathcal F_m $ and $\mathcal C_{m}$ $(m\in \mathbb N)$ of $D_2$
such that the interiors of
 atoms of $\mathcal F_m$ and  $\mathcal C_m$ have the form
${\color{red}x}B_s^{V}(h) B_{{\color{red}2r}}^{U}
$ and  $x B_r^{V} B_{s}^{U}(g)$
 respectively where
$s<\frac{1}{m}$, $g\in B_r ^U$ and $h\in B_r^V$. Furthermore the
 boundaries
of atoms of $\mathcal F_m$ and $\mathcal C_m$ can be taken to have
 $\nu$ measure
 zero.
\end{lemma}

The proof of Lemma \ref{lem:patition}
uses the linearized parameterization of $D_2$ introduced above.
 We omit the standard argument here.
We fix increasing partitions $\mathcal F_m$ 
and  $\mathcal C_m$ $(m\in \mathbb N)$
for $D_2$ as in Lemma \ref{lem:patition}.
 Let
$\mathcal F=\bigvee_{m\in \mathbb N}\mathcal F_m$,
 then it
is a countably generated  sigma-ring whose atoms {\color{red}contain} 
$yB_r^{U}$ {\color{red}for} $y\in D$.

According to the  property of the set $D$ in (\ref{eq:corec D})
 and the relationship
between leaf-wise measures and conditional measures
 discussed in Section \ref{sec:leafmeasure}, for any $m>0$,
\begin{equation}\label{limit}
\nu\big(\{x\in D:\nu_x^{\mathcal F}([x]_{\mathcal C_m})=1\}\big)>\sigma.
\end{equation}

Now we want to  reduce (\ref{limit}) to a similar equation
about $a_t\mu$ for some real number $t$
using Johnson and Rudolph's argument in Section 5 of
  \cite{JR95}.
Fix some $0<\epsilon <\min(\sigma,\frac{1}{16})$ and $m>0$, and set
$\mathcal C=\mathcal C_{m}$.
By (\ref{limit}) and the increasing martingale theorem,
for $N$ large enough and  $\mathcal A=\mathcal F_{N}$, we have
\begin{equation}
\nu\big(\{x\in D:\nu_x^{\mathcal A}([x]_{\mathcal C})>1-\epsilon/2\}\big)>\sigma-\epsilon/2.               \notag
\end{equation}

Since atoms of $\mathcal A$ and $\mathcal C$ have $\nu$ measure zero
 boundaries, there exists a positive integer $n$ such that for
\begin{equation}\label{defN}
\mathcal N=\{x\in D: \mu_{n,x}^{\mathcal A}([x]_{\mathcal C})>1-\epsilon\}
\end{equation}
 we have
$
\mu_n(\mathcal N)>\sigma-\epsilon.
$
Note that  $\mathcal A$ and $\mathcal C$ are finite partitions
and all the $x$ in some atom $ [y]_{\mathcal A\bigvee \mathcal C}$ give the same value
 for $\mu_{n,x}^{\mathcal A}([x]_{\mathcal C})$ which satisfies
 \[
 \mu_n([x]_{\mathcal A})\mu_{n,x}^{\mathcal A}([x]_{\mathcal C})=
 \mu_n([x]_{\mathcal A\bigvee \mathcal C}).
 \]
 We remark here that  $\mathcal N\cap [x]_{\mathcal A}$ is
either empty or an atom of $\mathcal A\bigvee \mathcal C$
 for any $x\in D$ since $\epsilon <\frac{1}{16} $.

For each  $x\in \mathcal N$, let
 \begin{equation}\label{Wx}
W(x)=
\{0\le t\le n : (a_t\mu)_x^{\mathcal A}([x]_{\mathcal C})>1-\epsilon^{1/2}\}.
\end{equation}
Let $W(x)^c=[0,n]\backslash W(x)$, 
then for any $x\in \mathcal N$ and
$t\in W(x)^c$,
\begin{equation}\label{eq:afterwx}
(a_t\mu)([x]_{\mathcal A}\backslash [x]_{\mathcal A\bigvee \mathcal C})
\ge \epsilon^{1/2}(a_t\mu) ([x]_{\mathcal A}).
\end{equation}
\begin{lemma}\label{lem:forany}
For any $x\in \mathcal N$, we have
\begin{equation}\label{claim1}
\frac{1}{n}\int_{ W(x)}(a_t\mu)([x]_{\mathcal A})
\,\mathrm d t
>(1-\epsilon^{1/2})\mu_n([x]_{\mathcal A}).
\end{equation}
\end{lemma}
\begin{proof}
Assume the contrary.
Then
\begin{eqnarray}
\mu_n([x]_{\mathcal A}\backslash[x]_{\mathcal A\bigvee \mathcal C})
&\ge&
\frac{1}{n}\int_{W(x)^c}
a_t\mu([x]_{\mathcal A}\backslash[x]_{\mathcal A\bigvee \mathcal C}) \,\mathrm d t
\notag\\
\mbox{(by (\ref{eq:afterwx}) )}\qquad&\ge &
  \frac{1}{n}\int_{W(x)^c}
\epsilon ^{1/2}(a_t\mu)([x]_{\mathcal A})\,\mathrm d t \notag \\
\mbox{(from the negation of (\ref{claim1}) )} \qquad&\ge & \epsilon\, \mu_n([x]_{\mathcal A}),\notag
\end{eqnarray}
which contradicts the definition of $\mathcal N$ in (\ref{defN}).
\end{proof}

Next we set
\begin{equation}\label{A}
A=\left\{0\le t \le n: a_t\mu{\big(}\{x\in \mathcal N:t\in W(x)\}{\big)}>(1-
\epsilon^{1/4})a_t\mu(\mathcal N)\right\}.
\end{equation}
Then for any $t\in A^c=[0, n]\backslash A$, we have
\begin{equation}\label{eq:at mu}
a_t\mu{\big(}\{x\in \mathcal N:t\not\in W(x)\}{\big)}\ge
\epsilon^{1/4}a_t\mu(\mathcal N).
\end{equation}
We remark here that for  fixed $t$, the set
 $\{x\in \mathcal N:t\not\in W(x)\}$
 is a union of atoms of $\mathcal A\bigvee \mathcal C$
and there is at most one of them in each atom of $\mathcal A$
since $\mathcal N\cap [x]_{\mathcal A}$ contains at most 
one atom of $\mathcal A\bigvee \mathcal C$.

\begin{lemma}\label{lem:A}
\begin{equation}\label{claim2}
\frac{1}{n}\int_{ A}a_t\mu(\mathcal N)
\,\mathrm d t > (1- {\color{red}2}\epsilon^{1/4})\mu_n(\mathcal N).               \notag
\end{equation}
\end{lemma}
\begin{proof}
Assume the contrary. Let
$\mathcal G=\{(x, t): x\in \mathcal N, t\not \in W(x)\}$.
Since  all the $x\in [y]_{\mathcal A\bigvee \mathcal C}$
have the same set $W(x)$,  we have that $\mathcal G$ is a union of
sets of the form $[x]_{\mathcal A\bigvee \mathcal C}\times
W(x)^c$. Therefore by (\ref{eq:at mu}) we have
\begin{eqnarray}
\frac{1}{n}\int_{ A^c}a_t\mu(\mathcal N) \,\mathrm d t
&\le&
\epsilon ^{-1/4}\frac{1}{n}\int_{ A^c} a_t\mu\big(\{x\in \mathcal N:t\not\in W(x)\}\big) \,\mathrm d t\label{pfclaim1} \notag\\
&\le& \epsilon ^{-1/4}\frac{1}{n}\sum_{[x]_{\mathcal A\bigvee \mathcal C}\subset \mathcal N}\int_{t\in W(x)^c}
a_t\mu([x]_{\mathcal A\bigvee \mathcal C})\,\mathrm d t. \notag
\end{eqnarray}
Since $\mathcal N\cap [x]_{\mathcal A}$ is either empty or a
single atom
of $\mathcal A\bigvee \mathcal C$, the right hand side of the last
inequality is
\begin{eqnarray*}
&\le& \epsilon ^{-1/4}\frac{1}{n}
\sum_{[x]_{\mathcal A\bigvee \mathcal C}\subset \mathcal N}\int_{t\in W(x)^c}
a_t\mu([x]_{\mathcal A})\,\mathrm d t             \\
(\mbox{by Lemma \ref{lem:forany}})\qquad &< & {\color{red}2}\epsilon ^{1/4}\mu_n(\mathcal N). \notag
\end{eqnarray*}
Now
a simple calculation gives us the inequality of the lemma.
\end{proof}

In view of Lemma \ref{lem:A}, there is some $t\in A$ such that
\[ a_t\mu(\mathcal N)\ge (1- {\color{red}2}\epsilon ^{1/4})\mu_n(\mathcal N).\]
Now we fix such a number $t$.
According to the definition of $\mathcal N$, $W(x)$  and $A$  in
(\ref{defN}), (\ref{Wx})
and (\ref{A}) respectively, we have
\begin{equation}
a_t\mu \big (\{x\in D:(a_t\mu)_x^{\mathcal A}([x]_{\mathcal C})>1-\epsilon^{1/2}\}\big)
\ge(1- {\color{red}2}\epsilon^{1/4})^2(\sigma-\epsilon).                       \notag
\end{equation}
Since $\big((a_t\mu)_x^{\mathcal A}\big)_y^{\mathcal F}=(a_t\mu)_y^{\mathcal F}$
for $ a_t\mu$  almost every $x\in D_2$ and $ (a_t\mu)_x^{\mathcal A}$  almost every $y\in D_2$,
we have
\begin{equation}\label{semifinal}
a_t\mu\big(\{x\in D:(a_t\mu)_x^{\mathcal F}([x]_{\mathcal C})>1-\epsilon^{1/4}\}\big)
\ge(1- {\color{red}2}\epsilon^{1/4})^3(\sigma-\epsilon).
\end{equation}

It is not hard to see that $(a_t\mu)_x^{\mathcal F}=
{\color{red}a_t}
\mu_{xa_{-t}}^{\mathcal Fa_{-t}}$.
Recall that $a_t$ normalizes the group $U$ according
to the assumption
 of Theorem \ref{thm:main}, so
 the atoms of $\mathcal F a_{-t}$ is still a section
 of $U$ orbit.
To get anything useful  we need that $a_{-t}$ does not change the shape
of atoms of $\mathcal F$ much
which is the only obstruction for the more general case discussed in
Remark \ref{rem:general case}.
This happens if $a_t$ commutes with $U$.
If not we will use the property that $\mu$ is invariant under the the diagonal
element  $a$  to choose some integer
$p$ so that the right translation by the inverse of $a_t a^p$ does not change the shape of atoms of $\mathcal F$
much.

Since the right translation of $a$ stretches foliations of $U$,
there exists an integer $p$ such that $a_t a^p$
does not stretch $U$ orbits, but $a_t a^{p+1}$ does.
Replace $\mu$ in (\ref{semifinal}) by $a^p\mu$ and    rewrite it we get
 \begin{equation}\label{eq:reduce}
\mu \big(\{x\in Da^{-p}a_{-t}:\mu_x^{\mathcal Fa^{-p}a_{-t}}([x]_{\mathcal Ca^{-p}a_{-t}}\big)>1-\epsilon^{1/4}\})
\ge(1- {\color{red}2}\epsilon^{1/4})^3(\sigma-\epsilon).
\end{equation}

According to the construction of $\mathcal F_m$ and 
$\mathcal C_m$ in Lemma \ref{lem:patition}
there is a constant $c\ge 2$ not depending on $t$ such that
\begin{equation}\label{eq:cor contained}
[x]_{\mathcal F a^{-p}a_{-t}\bigvee \mathcal Ca^{-p}a_{-t}}
\subset x B_{c/m}^U
\quad \mbox{and} \quad  xB_{r}^U\subset [x]_{\mathcal F a^{-p}a_{-t}}
\end{equation}
for any  $x\in Da^{-p}a_{-t}$.
In view of (\ref{eq:cor contained}), (\ref{eq:reduce})
and the relationship between conditional measures and
leaf-wise measures in Section \ref{sec:leafmeasure}, if we set
$$D_m=\left\{x\in X:\mu_x^U(B_{  c/m}^U)>
(1-\epsilon^{1/4})\mu_x^{U}(B_{  r}^U)\right\},$$
then
\begin{equation}
\mu(D_m)
\ge(1- {\color{red}2}\epsilon^{1/4})^3(\sigma-\epsilon).               \notag
\end{equation}

It is easy to see that $D_m\supset D_{m+1}$.
Since $\epsilon$ is independent of $m$, we may
let $ m$ go to infinity and conclude that
$\mu_x^U$ is trivial on a set whose $\mu$ measure is strictly bigger
than zero. This completes the proof of the following
\begin{lemma}\label{lem:conclu}
Under the notation of Theorem \ref{thm:main}, if the set
of  $x$ with $h_\nu(a, U)(x)=0$ has positive $\nu$ measure, then
$h_\mu(a, U)(x)=0$ on a positive $\mu$ measure set.
\end{lemma}
\begin{proof}[Proof of Theorem \ref{thm:main}]
Suppose the contrary, then Lemma \ref{lem:conclu} implies that
$h_\mu(a, U)(x)$ is  zero  on a positive   $\mu$ measure set.
Since $\mu$ is ergodic under the map $a$ and
the function $h_\mu(a, U)$ is
$a$-invariant, we have
 $h_\mu(a, U)(x)=0$ for $\mu$ almost every $x$. This
contradicts the assumption for $h_\mu(a, U)$.
\end{proof}

\section{Convergence of measures on quotients of
$SL_2(\mathbb R)\times L$}\label{sec:convergence}
In this section $G=SL_2(\mathbb R)\times L$ where $L$
is a product of semisimple algebraic groups over local fields of
characteristic zero,  $H<G$ is the $SL_2(\mathbb R)$ factor of $G$
and $a_t=\mathrm {diag}(e^t, e^{-t})\in H$.
We will use
 Theorem \ref{thm:main} to prove some
equidistribution results on homogeneous spaces using
measure rigidity in Lindenstruass \cite{L06}.
 We state a slightly variant version of
 Theorem 1.1 of \cite{L06}
 for the convenience of the reader.
\begin{theorem}[\cite{L06}, Theorem 1.1]
Let
$\Gamma$  be a discrete subgroup of $G$ such that
$\Gamma\cap L$ is finite. Suppose $\nu$ is a
probability measure on $X=\Gamma\backslash G$, invariant
under the multiplication from the right by elements of
the diagonal group flow $a_t$,
assume that
\begin{enumerate}
\item All ergodic components of $\mu$ with respect to the
flow $a_t$ have positive entropy.
\item $\mu$ is $L$-recurrent.
\end{enumerate}
Then $\mu$ is a linear combination of algebraic measures invariant
under $H$.
\end{theorem}

\begin{corollary}\label{cor:rigidity}
Let  $U<H$ be a subgroup of the unstable horospherical subgroup
of $a=a_{t_0}\times c\in G$ where $t_0\neq 0 $ and $c\in L$ is a
non-neutral class $\mathscr A$ element. Let $\Gamma $ be a discrete subgroup
of $G$ such that $\Gamma\cap L$ is finite.
Suppose
 $\mu$ is an $a$-invariant and ergodic probability measure on $\Gamma
\backslash G$
such that $h_\mu(a, U)>0$.
If for a sequence of real numbers $(T_n)$ with $T_n\to \infty$
$$\nu=\lim_{n\to \infty}\frac{1}{T_n}\int_0^{T_n} a_t\mu\, \mathrm dt,$$
  then
$\nu$ is a linear combination of algebraic measures invariant under
$H$.
\end{corollary}
\begin{proof}
Assume $\nu$ is nonzero, otherwise there is nothing to prove.
It is easy to see that $\nu$ is invariant under the subgroup
generated by $a_t$ for $t\in \mathbb R$ and $c$.
Since $c$ is from the class $\mathscr A$ and non-neutral, the cyclic
group generated by $c$ is unbounded.
 It follows from the invariance of $\nu$ under $c$  that $\nu$
is recurrent for the factor
$L$. Note that
 the right translation of $a_{t_0}$ also stretches
$U$ orbits. The conclusion of Theorem \ref{thm:main} implies that
for $\nu$ almost every $x$,
 \[
 h_{\nu}(a_{t_0}, U)(x)= h_{\nu}(a, U)(x)>0.
 \]
 So almost all the ergodic components
of $\nu$ with respect to the flow $a_t$
have positive entropy.
 Now the assumptions of
Theorem 1.1 of Lindenstruass \cite{L06} are satisfied
and the conclusion there is what we want.
\end{proof}

We say that
two  normal subgroups  $G_1 $ and $G_2$ of $G$ are  complementary
if $G=G_1G_2$ and $G_1\cap G_2$ is finite.
A lattice  $\Gamma$ of $G$ is said to be irreducible
if there are no {\color{red}infinite} complementary subgroups $G_1$ and
$G_2$ of $G$ such that
 $(G_1\cap \Gamma)\cdot (G_2\cap \Gamma)$
 has finite index in $\Gamma$.
If $L$ is a product of  simply connect algebraic groups
in the sense of \cite{M} I.1.4.9 or connected real Lie groups, then
 the natural projection of an irreducible lattice
  $\Gamma$ to $L$ has dense image. We remark here that
 the group $SL_m$ is simply connected.

\begin{lemma}\label{prop:uniq ergo}
Suppose $\Gamma$ is a lattice of $H
\times L_1$ where $L_1$ is a closed subgroup
$L$ and the natural projection of $\Gamma$
to $L_1$ has dense image.
Then $H$ acts uniquely ergodicly on
$\Gamma\backslash H\times L_1$ with  the probability Haar measure
as the unique $H$ invariant probability measure.
\end{lemma}
The proof is standard and we omit it here.

\begin{corollary}\label{cor:limit}
Under the notation and assumption of Corollary \ref{cor:rigidity},
  let $L_1$ be the closure  in $L$ of the image of
  $\Gamma$ under the natural projection map.
Suppose in addition that
  $\Gamma$ is an irreducible lattice and
$\mu$ is supported on
   $X=\Gamma\backslash H\times L_1$,  then
$\nu=\nu(X)
m_{X}$. Moreover
 if
 $\mu$ has non-escape of mass for the flow $a_t$,
 then
 \begin{equation}\label{eq:cor limit}
\lim_{T\to \infty}\frac{1}{T}\int_0^T a_t \mu\, \mathrm d t
=m_{X}.
\end{equation}
\end{corollary}
\begin{proof}
Corollary \ref{cor:rigidity} implies that $\nu$
is invariant under the first factor $H$
of $G$.
Since $X$ is closed and $\mu$ is supported on $X$,
the measure $\nu$ is supported on $X$.
Since the natural projection of
  $\Gamma$ to $L_1$ has dense image,
Lemma \ref{prop:uniq ergo} implies
that $\nu=\nu(X)m_{X}$.

If $\mu$ has non-escape of mass, then $\nu=m_{X}$.
Since the sequence  $(T_n)$ is arbitrary, (\ref{eq:cor limit})
follows.
\end{proof}


\begin{example}
Let $k$ be a totally real quadratic number field
with ring of integers $\mathfrak o$. Let
$\tau$
be the non-trivial Galois automorphism of $k$ over $\mathbb Q$.
For each $g\in SL_2(\mathfrak o)$ let $g^\tau$ be the componentwise
conjugation, then  $SL_2(\mathfrak o)$ embeds in
$G=SL_2(\mathbb R)\times SL_2(\mathbb R)$ through  the map
$g\to (g, g^{\tau})$ as an irreducible lattice.
Suppose $\gamma\in SL_2(\mathfrak o)$ is a diagonal matrix
 whose eigenvalues are not root of unity.

Take $X=SL_2(\mathfrak o)\backslash G$ and
 $\vartheta_X=SL_2(\mathfrak o)\in X$.
It is not hard to see that  the periodic orbit
$
\vartheta_X\left(
\begin{array}{cc}
1 & 0 \\
\mathbb R^2 &1
\end{array}
\right)
$
is invariant under $\gamma$. Moreover the action of $\gamma$
on this periodic orbit is isomorphic to a hyperbolic action
on $\mathbb T^2$.
We claim that if we put a positive entropy ergodic measure
of $\mathbb T^2$ on this periodic orbit to get a measure $\mu$,
then $\mu$ has positive entropy contribution along some
one-dimensional unipotent orbit. The idea is that there are two directions
for this $\mathbb R^2$ orbit and one of them is stretched
 by $\gamma$ and the
other is contracted.
\end{example}

If $L=SL_2(\mathbb R)$, then
we could give a partial answer to
Conjecture 9.2 of Einsiedler \cite{E}.
\begin{corollary}\label{cor:conjec of Ein}
Let $\mu$ be an ergodic probability measure on
the irreducible quotient of $
\Gamma\backslash SL_2(\mathbb R)\times SL_2(\mathbb R)$
with respect to the group generated by
\[
b_t=\mathrm {Id}\times \mathrm{diag}(e^t, e^{-t})\quad t\in \mathbb R.
\]
If $\mu$ has
 positive entropy with respect to $b_t$ when $t\neq 0$ and has non-escape of mass
 for the flow $a_t$, then
 \[
\lim_{T\to \infty}\frac{1}{T}\int_0^T a_t \mu\, \mathrm d t
=m_{\Gamma\backslash G}.
\]
\end{corollary}
\noindent
The proof is similar to that of Corollary \ref{cor:limit}
but uses a variant of Theorem \ref{thm:main}.
In above corollary the measure $\mu$ is ergodic under the
flow $b_t$, whereas in Theorem \ref{thm:main} the given
measure $\mu$ is invariant {\color{red}and ergodic with respect to a cyclic group}.
It is not hard to see that the proof of Theorem \ref{thm:main}
still works  in this case to give positive entropy of  ergodic components
of the limit measure.

\section{Convergence of measures on
$SL_2(\mathbb Z)\backslash SL_2(\mathbb R)$}
\label{sec:example}

Let $n>1$ be a positive integer and let $\lambda$ be a probability
measure on $[0,1]=\mathbb Z\backslash \mathbb R$ such that
$\lambda$ is $\times n$-invariant and ergodic with positive
entropy.
Recall that for $s\in \mathbb R$,
\begin{equation}\label{eq:ux}
u(s)=
\left(\begin{array}{cc}
1 & 0 \\
s & 1
\end{array}
\right)\in SL_2(\mathbb R).               \notag
\end{equation}
Let $Y=SL_2(\mathbb Z)
\backslash SL_2(\mathbb R)$ and $\vartheta _Y=SL_2(\mathbb Z)\in Y$.
Let
$\tilde \mu$ be the measure on $Y$ such that for any $f\in C_c(Y)$
\begin{equation}\label{eq:induceY}
\int_Y f \,\mathrm d \tilde\mu=\int_0^1 f(\vartheta_Y u(s))\,\mathrm d\lambda(s).
\end{equation}
We also use $\times n$ to denote the measure preserving map on
$(Y, \tilde \mu)$ that sends $\vartheta_Y u(s)$ to $\vartheta_Y u(ns)$.

To guarantee that there is  non-escape of
mass under diagonal  flow $a_t=(e^t, e^{-t})$
 we need some additional assumption on the measure $\lambda$.
One possible choice is to assume that $\lambda$ is friendly, which is defined
in \cite{KLW}. Examples of friendly measures consist of Hausdorff 
measures supported on  self-similar sets such as Cantor sets.
The following theorem is a general version of Theorem \ref{thm:intro2}.
\begin{theorem}\label{thm:equi:sl2}
Let $\lambda$ be  friendly
 probability
measure on $[0,1]$. Suppose that $\lambda$ is
$\times n$-invariant  and ergodic
  with positive
entropy.
Let
$\tilde \mu$ be the measure on $Y$ induced from
$\lambda$ by (\ref{eq:induceY}), then
\[
\lim_{T\to \infty}\frac{1}{T}\int_0^T a_t\tilde \mu \, \mathrm d t=m_Y.
\]
\end{theorem}
To prove Theorem \ref{thm:equi:sl2}, we need first prove
Theorem \ref{thm:intro0}.
For a field $k$ of characteristic zero, let $PGL_2(k)$ be the $k$ points of the
algebraic group $PGL_2$. The natural map
$\pi_k: GL_2(k)\to PGL_2(k)$ is surjective and
the kernel consists of diagonal matrices of $GL_2(k)$.
For a ring $R\subset k$, we use $PGL_2(R)$ to denote the image
of $GL_2(R)$ under $\pi_k$.

Suppose  the
prime decomposition of $n=p_1^{\sigma_1} \cdots p_l^{\sigma_l}$.
Let
\[
S=\{\infty, p_1, \ldots, p_l\},\quad
\mathbb Q_S=\prod_{v\in S}\mathbb Q_v,\quad
K=\prod_{p\in S, p\neq \infty}PGL_2(\mathbb Z_{p}).
\]
Take $\Gamma=PGL_2(\mathbb Z[\frac{1}{p_1}, \ldots, \frac{1}{p_l}])$
 which sits
 diagonally in the group
$G=PGL_2(\mathbb Q_S)$
as an irreducible lattice.
It is not hard to see that
$$X=\Gamma\backslash G=
PGL_2(\mathbb Z) \backslash
PGL_2(\mathbb R)\times K.$$
We can define a factor map
$
\eta: X\to Y
$ which sends $(g, h)\mod \Gamma$ in $X$ with $g\in PGL_2(\mathbb R)$
and $h\in K$ to  $g\mod SL_2(\mathbb Z)$ in $Y$.
For $\mathbf v\in  \mathbb Q_S$, let
$u(\mathbf v)=\left(
\begin{array}{cc}
1 & 0 \\
\mathbf v & 1
\end{array}
\right)$.
 Let $V$ be the unipotent subgroup
of $G$ generated by $u(\mathbf v)$ for
 $\mathbf v\in  \mathbb Q_S$.

\begin{lemma}\label{lem:entropyto contribution}
Let $\mu$ be an $a$-invariant and ergodic  probability measure on $X$ with
positive entropy. Suppose  $\mu$ is supported on $\vartheta_X u(\mathbb Q_S)$,
then
$h_\mu(a, U)>0$
where $U=\{u(s):s\in \mathbb R\}$.
\end{lemma}
\begin{proof}

Let 
$\Lambda=V\cap \Gamma$, then $\Lambda$ is a lattice of $V$
and $\Gamma V\cong \Lambda \backslash V$ is
an $a$-invariant compact subset of $X$.
The measure $\mu$ can be viewed as a probability measure on
 $ \Lambda \backslash V$.
 Note that there is only one expanding foliations for
 $a$ on $\Lambda \backslash V$,
 namely, the orbits of
 $U$. Since $h_\mu(a)>0$ and $\mu$ is  ergodic under the map $a$,
 we have
 $h_\mu(a, U)>0$.
\end{proof}

\begin{proof}[Proof of Theorem \ref{thm:intro0}]
Let $L=\prod_{v\in S, v\neq \infty}SL_2(\mathbb Q_v)$.
It is not hard to see that there is an irreducible lattice
$\Gamma_0$ of $G_0=SL_2(\mathbb R)\times L$ such that
$X\cong X_0=\Gamma_0\backslash G_0$ under the natural map
$G_0\to G$. Let $L_1$ be the closure in $L$ of the image
of $\Gamma_0$ under the natural projection map.
Let
$b=\mathrm {diag}(\sqrt n, 1/\sqrt n)\times \mathrm {diag}(n,1)^l\in G_0$
and let $\mu_0$ be the measure on $X_0$ correspond to $\mu$ on $X$.
Then
the system $(X, \mu , a, a_t)$ is isomorphic to the system
$(X_0,\mu_0, b, a_t )$.

Note that $\mu_0$ is supported on
$\Gamma_0 (SL_2(\mathbb R)\times L_1)$.
The assumption that $h_\mu(a)>0$ implies $h_{\mu_0}(b, U)>0$ according
to Lemma \ref{lem:entropyto contribution}. Since $\mu$ has non-escape
of mass for the flow $a_t$, the  conclusion follows from
Corollary \ref{cor:limit}.
\end{proof}

Now we return to the proof of Theorem \ref{thm:equi:sl2}.
The measure $\lambda$ on $[0,1]$ induces a measure $\mu_1$
on $X$ such that
 for
any $f\in C_c(X) $,
\[
\int_{X} f \,\mathrm d \mu_1=\int_0^1 f(\vartheta_X u(s))\,\mathrm d\lambda(s).
\]
Then the pushforward
 $\eta\mu_1=\tilde \mu$.
Note that
$\mu_1$ is not invariant for the right translation of
$a=\mathrm{diag}(n,1)^{l+1}\in G$.
But we can still guarantee   that
there is an $a$-invariant and ergodic
 measure on $X$ that projects to $\tilde \mu$.

\begin{lemma}\label{lem:key}
There exists an $a$-invariant and ergodic probability measure $\mu$ on $X$
supported on $\vartheta_X u(\mathbb Q_S)$ where $\vartheta_X=\Gamma\in X$
 such
that $\eta \mu=\tilde \mu$ and $h_\mu(a)>0$.
\end{lemma}

\begin{proof}
Let $\mu_1$ be as above and let
 $\mu_2$
be a limit point in the weak$^*$ topology for the sequence
$\frac{1}{N}\sum_{i=0}^{N-1} a^i \mu_1 $.
It is clear that $\mu_2$ is $a$-invariant.
We claim that $\mu_2$ is a  probability measure.
 Since
$
\Lambda =V \cap
\Gamma
$
 is a lattice of $V$.
It follows that
 $\Gamma V\cong \Lambda \backslash V$ is
an $a$-invariant compact subset of $X$. Since each $a^i\mu_1$ is supported
on  $\Gamma V$, the measure $\mu_2$ is a probability measure.

Now $\eta$ is a factor map from $(X, \mu_2, a)$ to
$(Y, \tilde \mu , \times n)$ and $h_{\tilde \mu}(\times n)>0$,
so $h_{\mu_2}(a)>0$. Take $\mu $ to be some ergodic component of
$\mu_2$ with positive entropy, then $\mu$ satisfies the
requirement of the lemma.
\end{proof}

\begin{proof}[Proof of Theorem \ref{thm:equi:sl2}]
By Lemma \ref{lem:key}, there exists an $a$-invariant and
ergodic probability measure
$\mu$ on $X$
such that $\eta \mu=\tilde \mu$ and $h_\mu(a)>0$.
Moreover,
the following diagram
\[
\xymatrix{
(X,\mu) \ar[r]^{a_t} \ar[d]_{\eta} & (X, \mu) \ar[d]^{\eta} \\
(Y,\tilde \mu) \ar[r]_{a_t}  & (Y,\tilde \mu)
}
\]
commutes.
 Since $\lambda $ is
friendly, there is non-escape of mass for $\tilde \mu$ under the flow
$a_t$. The commutative diagram above and the fact
that fibers of $\eta$ are compact imply that there is non-escape
of mass for $\mu$ under the flow $a_t$. Therefore $\mu$ satisfies
all the assumptions of Theorem \ref{thm:intro0}. So $\mu$
is equidistributed
on average for the flow $a_t$ with respect to an $H$ invariant
homogeneous measure.
Therefore $\tilde \mu$ is equidistributed on average for the flow
$a_t$ on $Y$.
\end{proof}

\section{Convergence of measures on
 $SL_3(\mathbb Z)\backslash SL_3(\mathbb R) $}
 \label{sec:converSL3}

 In the next two sections let $G= SL_3(\mathbb R)$,
 $\Gamma= SL_3(\mathbb Z)$ and
  $Z=\Gamma\backslash G$.
  Other notations are the same as in Theorem \ref{thm:intronew}
  which we prove right now.
 Let $\nu=\lim \mu_n$ and we want to show that $\nu$ is a
 scalar multiple of the probability Haar measure.
 Without loss of generality we assume that $\nu\neq 0$.

 For the hyperbolic matrix $\gamma\in SL_2(\mathbb Z)$, let $
g_\gamma=\left(
\begin{array}{cc}
\gamma & 0 \\
0 &      1
\end{array}
\right).
$
Then $\mu$ is $g_\gamma$-invariant and  has positive entropy.
Similar to Lemma \ref{lem:entropyto contribution}, we have
\begin{lemma}\label{lem:entropyto contriSL3}
For the probability measure  $\mu$
 in
Theorem \ref{thm:intronew}, then there exists
a one dimensional subgroup of the unstable horospherical
subgroup of $g_\gamma$ such that  $h_\mu(g_\gamma, U)(x)>0$
for $\mu$ almost every $x$.
\end{lemma}
It follows from Theorem \ref{thm:main} that the average
 measure $\nu$ in Theorem \ref{thm:intronew}
 has positive entropy contribution along $U$ orbits almost surely.
So
a typical ergodic component of $\nu$ with respect to
the group $\Lambda$ generated by $a_t$ and $g_\gamma$,
say probability measure $ \nu_1$,
satisfies  $h_{ \nu_1}(g_\gamma)>0$.

Let $A$ be the Cartan subgroup of $G$ containing $\Lambda$, then
\begin{equation}\label{eq:ergo compo}
\nu_2=\int _{\Lambda\backslash A} a \nu_1
\, \mathrm d m_{\Lambda\backslash A}(a)
\end{equation}
is an $A$-invariant probability measure with all the
$A$-ergodic components having positive entropy with respect
to $g_\gamma$.
Now we need to use the measure rigidity theorem
 of Einsiedler,Katok and Lindenstrauss \cite{EKL}.
\begin{theorem}[\cite{EKL}]\label{thm:EKL}
Let $\mu$ be an $A$-invariant and ergodic probability measure on $Z$,
Assume that there is some one-parameter subgroup of $A$ which
acts on $Z$ with positive entropy. Then $\mu$ is the unique
$SL_3(\mathbb R)$ invariant measure.
\end{theorem}

In view of (\ref{eq:ergo compo}) and
the positive entropy of $\nu_1$ with
respect to $g_\gamma$, we have that all
the ergodic component of $\nu_2$ with respect to
the   Cartan subgroup $A$ have positive entropy for
the map $g_\gamma$. Theorem \ref{thm:EKL} implies
that $\nu_2=m_Z$.
 By (\ref{eq:ergo compo}),
\[
\int _{\Lambda\backslash A} h_{a \nu_1}(g_\gamma)
\, \mathrm d m_{\Lambda\backslash A}(a)=h_{m_Z}(g_\gamma).
\]
Since the group $A$ is abelian, we have
$
h_{a \nu_1}(g_\gamma)=h_{ \nu_1}(g_\gamma)$
Therefore, $h_{ \nu_1}(g_\gamma)=h_{m_Z}(g_\gamma)$.

According to the results in Section 9 of \cite{MT},
 $\nu_1$ and $m_Z$ have the same
entropy  if and only if they are equal. Since $\nu_1$ is typical in the ergodic
decomposition of $\nu$, the measure $\mu$
 must be a scalar multiple of $m_Z$.
This completes the proof of Theorem \ref{thm:intronew}.

 \begin{corollary}\label{cor:sl3equi}
Under the notation and assumption of Theorem \ref{thm:intronew},
if $ \mu$ has non-escape of mass with respect to the
flow $a_t$,
 then
 \begin{equation}\label{eq:in intronew}
\lim_{T\to \infty}\frac{1}{T}\int_0^T
a_t \mu
 \, \mathrm d t
= m_Z.
\end{equation}
\end{corollary}

\section{Applications to Diophantine approximation}
\label{sec:Diophantine}
Let $\|\cdot\|$ be the sup norm on the Euclidean space $\mathbb R^n$.
 We say that
$\mathbf v\in \mathbb R^2$ is Dirichlet improvable
(or DI for short), if there exists $0<\sigma<1$ such that for sufficiently large integer $N$, there
are integer $n$ and vector $\mathbf w$ such that
\begin{equation}\label{eq:DI}
\|n\mathbf v-\mathbf w\|<\sigma N^{-1}\quad \mathrm{and}
\quad 0<|n|<\sigma N^2.
\end{equation}
In this case we say that $\mathbf v$ can be $\sigma$-improved.
  Let\[DI_\sigma=\{\mathbf v\in \mathbb R^2: \mathbf v \mbox{ can be }\sigma
\mbox{-improved.}\}
\quad \mbox{and}\quad DI=\bigcup_{0<\sigma<1} DI_\sigma.
\]

Let the notation be as in Section \ref{sec:converSL3}. Let  $\pi:G\to Z$
 be the natural projection.
The dynamical approach to this Diophantine
approximation problem is through the map
\[
u:\mathbb R^2\to G \quad \mbox{where} \quad u(\mathbf v)=
\left(
\begin{array}{cc}
\mathrm{Id_2} & 0 \\
\mathbf v & 1
\end{array}
\right),
\]
then study the orbit of $\pi( u(\mathbf v))$. Recall that
$Z$ can be identified
with the set of unimodular lattices of $\mathbb R^3$ and the
correspondence is $\pi(g)\to \mathbb Z^3 g$ for $g\in G$.
Let
$K_\sigma$ be the set of unimodular lattices
in $\mathbb R^3$ whose shortest
nonzero vector has norm bigger than or equal to  $ \sigma$.
Then $\mathbf v\in DI_\sigma$
implies that for $\color{red}t$ large
enough $\pi (u(\mathbf v))a_t\not \in K_{\sqrt\sigma}$.

\begin{theorem}\label{thm:DI}
Let $\gamma\in SL_2(\mathbb Z)$ be a hyperbolic matrix and
let $\lambda$ be a $\gamma$-invariant and
ergodic  probability measure on $[0, 1]^2$ {\color{red}with positive entropy}. Let $\mu=(\pi u)\lambda$
and
$a_t=(e^t, e^t, e^{-2t})$. If
$
\lim_{T\to \infty}\int_0^Ta_t \mu\, \mathrm d t\neq 0,
$
 then
 Dirichlet's theorem is not improvable   $\lambda$ almost surely.
 \end{theorem}
\begin{proof}
Assume that  $\lambda(DI)>0$, then
for some $0<\sigma<1$ we have
 $\lambda(DI_\sigma)=\tau>0$.
Let $(T_n)$ be a sequence of real numbers  such that
$T_n\to \infty$ and
\[
\lim_{n\to \infty}\frac{1}{T_n}\int_0^{T_n} (a_t\pi u)\lambda \,
\, \mathrm d t
 =\nu  .
\]
By passing to a subsequence, we may assume that
$$\lim_{n\to \infty}\frac{1}{T_n}\int_0^{T_n} (a_t \pi u)\lambda|_{DI_\sigma} \,
\, \mathrm d t=\nu_1.$$
Since elements of  $\pi (u (DI_\sigma))$ eventually do not intersect
 $K_{\sqrt\sigma}$ under the flow $a_t$,
we have $\nu_1(K_{\sqrt\sigma})=0$. We remark here that $K_{\sqrt\sigma}$
is a neighborhood of $\mathbb Z^3$.

 Take
$g_\gamma=\left(\begin{array}{cc}
\gamma & 0     \\
0   &    1
\end{array}
\right)\in G$, then
$$\lim_{n\to \infty}\frac{1}{T_n}\int_0^{T_n} (a_t\pi u)
\lambda|_{\gamma(DI_\sigma)} \,
\, \mathrm d t=g_\gamma\nu_1.$$
Since $m_Z$ is ergodic for the flow $a_t$, it is
the only possible ergodic component of $\nu$ for the flow $a_t$
according to  Theorem \ref{thm:intronew}.
Note that $g_\gamma\nu_1$ is invariant for the flow $a_t$ and
it vanishes on some open subset. Therefore $g_\gamma\nu_1=0$
Since $\lambda$ is $\gamma$-ergodic, we see that $\nu=0$.
Since the sequence $(T_n)$ is arbitrary, this implies
$\lim_{T\to \infty}\int_0^Ta_t \mu\, \mathrm d t= 0$
which contradicts to the assumption.  Therefore $\lambda (DI)=0$.
\end{proof}

Now we give some examples of $\lambda$ such  that
there is some mass left  for the flow $a_t$ and prove
Theorem
\ref{thm:main diophantine}.
 Let $B_r(\mathbf v)$ with $\mathbf v\in\mathbb R^2$
be the ball of radius $r$ centered at $\mathbf v$ under the
norm $\|\cdot\|$. We need the
  following well-known result:
\begin{proposition}\label{prop:exact dime}
Let $\gamma\in SL_2(\mathbb Z)$ be a hyperbolic matrix.
Suppose $\lambda$ is a $\gamma$-invariant and ergodic
 probability measure on $\mathbb T^2$ such that
 $h_\lambda(\gamma)=c\,h_m(\gamma)$ where $m$ is
 the Lebesgue measure and $0\le c\le 1$, then $\lambda$
 has exact dimension $2c$, that is \[
 \lim_{\epsilon\to 0}
 \frac{\log \mu(B_\epsilon(\mathbf v))}{\log \epsilon}=2c
 \]
 for $\lambda$ almost every $\mathbf v\in \mathbb R^2$.
\end{proposition}

Let $G^+$ and  $G^-$ be the unstable and stable  horospherical
subgroup of $a_1$ in $G$ and let $C$ be
the centralizer of $a_1$.  We fix a right invariant Riemannian
metric on  $G^+ $ and $G^{-}C$ respectively, and denote
by $B_r^{G^+}$ and $B_r^{G^-C}$ the ball of radius $r$
centered at the identity with respect to the induced metric.
\begin{definition}
 For a probability measure $\mu$ on $Z$ we say that $\mu$ has dimension at
least $d$ in the unstable direction of $a_1$ if for any $\delta > 0$
 there exists $\kappa > 0$ such that for
any $\epsilon\in  (0, \kappa)$ and for any $\sigma\in (0, \kappa)$ we have
\[
\mu(x B_{\epsilon}^{G^+}B_\sigma^{G^-C})\ll_\delta \epsilon ^{d-\delta} \quad  \mbox{for  any } x\in Z.
\]

\end{definition}

Let $\lambda$ be as in Proposition \ref{prop:exact dime},
then $\mu=(\pi u)\lambda$ has at least dimension $2c$
in the unstable direction of $a_1$. If $c>\frac{2}{3}$, then
Theorem 1.6 of Einsiedler and  Kadyrov \cite{EK}
shows that there is some mass
left.

\begin{theorem}[\cite{EK}]\label{thm:EK}
For a fixed d, let $\mu$ be a probability measure on $Z$
of dimension at least
$d$ in the unstable direction. Let $(T_n)$ be
a sequence of positive real numbers such that
\[
\nu=\lim_{n\to\infty}{\color{red}\frac{1}{T_n}}\int_0^{T_n} a_t\mu\,\mathrm d t.
\]
Then $\nu(Z)\ge \frac{3}{2}(d-\frac{4}{3})$.
\end{theorem}

Now Theorem \ref{thm:main diophantine} follows form
Theorem \ref{thm:DI},
Proposition  \ref{prop:exact dime} and Theorem \ref{thm:EK}.

\end{document}